\documentclass[11pt,a4paper]{amsart}\usepackage{fullpage}
\usepackage{amssymb,amscd}

\newtheorem{thm}{Theorem}[section]
\newtheorem{prop}[thm]{Proposition}
\newtheorem{lem}[thm]{Lemma}
\newtheorem{cor}[thm]{Corollary}

\theoremstyle{definition}
\newtheorem{defn}[thm]{Definition}

\newtheorem{remark}[thm]{Remark}

\numberwithin{equation}{section}

\def\leq{\leqslant}
\def\geq{\geqslant}
\def\({\left(}
\def\){\right)}
\def\:{\colon}
\def\.{\cdot}
\def\o{\circ}
\def\ds{\displaystyle}
\def\phi{{\varphi}}
\def\epsilon{\varepsilon}
\def\lra{\longrightarrow}

\def\k{\Bbbk}
\def\Z{{\mathbb Z}}
\def\C{{\mathbb C}}
\def\Q{{\mathbb Q}}

\def\H{{\mathbb H}}

\def\CP{\C\mathrm P}
\def\HP{\H\mathrm P}

\def\MU{MU}
\def\MSp{MSp}
\def\BP{BP}
\def\BU{BU}
\def\Sp{Sp}
\def\BSp{BSp}
\def\U{\mathrm U}
\def\evMU{{}^{\mathrm{odd}}\MU}

\def\ideal{\triangleleft}
\def\homeq{\simeq}
\def\iso{\cong}
\def\d{\mathrm d}

\def\L{\mathrm L}

\def\mapsto>#1>#2>{\mapstochar\joinrel{\xrightarrow[#2]{#1}}}
\def\mapsfrom<#1<#2<{\xleftarrow@[#2]{#1}\joinrel{\mapstochar}}
\def\incl>#1>#2>{\lhook\joinrel\xrightarrow[#2]{#1}}
\def\rincl<#1<#2<{\xleftarrow@[#2]{#1}\joinrel\rhook}
\def\lcni<#1<#2<{\xleftarrow@[#2]{#1}\joinrel\rhook}

\DeclareMathOperator{\Hom}{Hom}
\DeclareMathOperator{\id}{id}
\DeclareMathOperator{\im}{im}
\DeclareMathOperator{\Invol}{Invol}

\def\f{\mathbf f}
\def\v{\mathbf v}
\def\T{\mathrm T}
\def\M{\mathrm M}
\def\X{\mathrm X}
\def\B{\mathrm B}
\def\Smash{\wedge}
\def\FGL{\mathbf{FGL}}
\def\evFGL{{{}^{\mathbf{odd}}\mathbf{FGL}}}
\def\Alg{\mathbf{Alg}}
\def\Times#1{\mathop{\times}_{#1}}
\def\oTimes#1{\mathop{\otimes}_{#1}}
\def\Plus#1{\mathop{+}_{#1}}

\def\pev{p^{\mathrm{odd}}}
\def\evS{{}^{\mathrm{odd}}\mathrm S}
\def\G{\mathrm{G}}
\def\Godd{\G_{\mathrm{odd}}}
\def\Gev{\G_{\mathrm{ev}}}
\def\L{\mathrm{L}}
\def\Lodd{\L_{\mathrm{odd}}}
\def\iff{\quad\Longleftrightarrow\quad}
\def\CPi{\CP^\infty}
\def\tilde{\widetilde}
\def\Dia#1{\mathop{\diamond}_{#1}}

\def\dsum{\ds\sum}

\title[$\MSp$ localized away from $2$ and odd formal group laws]
{$\MSp$ localized away from $2$ and odd formal group laws}
\author{Andrew Baker \& Jack Morava}
\address{University of Glasgow, Glasgow G12~8QW, Scotland.}
\email{a.baker@maths.gla.ac.uk}
\urladdr{http://www.maths.gla.ac.uk/$\sim$ajb}
\address{Johns Hopkins University, Baltimore, MD~21218, USA.}
\email{jack@math.jhu.edu}
\urladdr{http://mathnt.mat.jhu.edu/mathnew/Faculty/HomePages/jmorava.html}
\thanks{The authors would like to acknowledge the support of
l'Institute des Hautes \'Etudes Scientifiques, Osaka Prefecture, 
the EU and EPSRC and probably others, whilst parts of this work 
were undertaken since ca 1992. 
}
\keywords{symplectic cobordism, formal group law, symmetric space}
\subjclass[2000]{primary 55N22, 57R77, 57R90; secondary 14L05}

\begin{document}
\begin{abstract}
We investigate the relationship between complex and symplectic
cobordism localized away from the prime~$2$ and show that these
theories are related much as a real Lie group is related to its
complexification. This suggests that ideas from the theory of
symmetric spaces might be used to illuminate these subjects. 
In particular, we give an explicit equivalence of ring spectra
\[
\MSp[1/2]\Smash\Sp/\U_+\simeq\MU[1/2]
\] 
and deduce that $\MU[1/2]$ is a wedge of copies of $\MSp[1/2]$. 
We discuss the implications for the structure of the stable 
operation algebra $\MSp[1/2]^*\MSp[1/2]$ and the dual cooperation 
algebra $\MSp[1/2]_*\MSp[1/2]$. Finally we describe some related 
Witt vector algebra and apply our results to the study of formal 
involutions on the category of formal group laws over a 
$\Z[1/2]$-algebra.
\end{abstract}
\maketitle

\section*{Introduction}
Cobordism theory is a part of topology in which geometry encounters
homotopy theory in a particularly transparent way, and it is striking
that the algebra which results is often remarkably interesting. This
paper is concerned with complex and symplectic cobordism; it is
well-known that the latter is extremely complicated at the prime~$2$,
but we show that at odd primes, these two theories are related much
as a real Lie group is related to its complexification. Ideas from
the theory of symmetric spaces can thus be used as a guide to an
area where the structure is otherwise perplexing. The following
are our three main results.

Our main topological results are the following. Theorem~\ref{thm:2.3} 
shows that there is an equivalence of ring spectra $\MSp[1/2]\lra\evMU$, 
where $\evMU$ is a certain ring spectrum constructed with the aid of 
an idempotent map of ring spectra $\MU[1/2]\lra\MU[1/2]$, produced 
using formal group theory; the proof is directly modelled on the 
approach to splitting the $p$-localization of $\MU$ used by 
Quillen~\cite{Qu,JFASHGH}, and is consistent with it for all odd 
primes~$p$. In Theorem~\ref{thm:2.2}, the homotopy ring $\evMU_*$ 
is characterized as universal for a class of formal group laws over 
$\Z[1/2]$-algebras that we call \emph{odd}. In Theorem~\ref{thm:2.5} 
we show that there is an equivalence of the ring spectra $\MU[1/2]$ 
and $\MSp[1/2]\Smash\Sp/\U_+$.

In Theorem~\ref{thm:3.4}, we give an algebraic characterisation
of the ring 
\[
\MSp[1/2]_*\MSp[1/2]\iso\evMU_*\evMU
\] 
using an identification of the second Hopf algebroid in terms 
of isomorphisms of odd formal group laws; this is related to 
work of Buchstaber~\cite{VMB}, the precise connection being stated 
in Proposition~\ref{prop:3.3}. We obtain a related identification 
of the stable operation algebra in the cohomology theory 
$\MSp[1/2]^*(\ )$ with a subalgebra of that for $\MU[1/2]^*(\ )$; 
in keeping with work of Buchstaber \& Shokurov
and Morava~\cite{BS,Mor}, this realizes an action of the even positive
half of the formal vector fields Lie algebra associated to the
diffeomorphism group of the circle. This seems to have connections
with the KdV equation and Witten's topological field theory, as 
well as the work of Katsura, Shimizu \& Ueno~\cite{KSU}.

We reformulate our topological results in scheme theoretic language
and relate them to the scheme of formal involutions viewed as a sort
of `symmetric space'. We also discuss some Witt vector algebra related
to both the algebraic and topological aspects of our work.

We would like to thank Victor Buchstaber and Nigel Ray for teaching
us many things relevant to this work; also the referee of a previous
version for detailed suggestions on improving the exposition.

\section{Some calculations with the $[-1]$-series of a formal group law}
\label{sec:1}

Throughout this section we will work with a graded commutative and
unital ring $R=R_*$, which we assume is torsion free, hence can be
embedded in its rationalization $R\Q=R\otimes\Q$. Let
$F(X,Y)=\dsum_{i,j}a^F_{i,j}X^iY^j$ denote a commutative $1$-dimensional
formal group law over $R$, where $a^F_{i,j}\in R_{2i+2j-2}$. Over
the ring $R\Q$, $F$ admits unique logarithm $\log^F(X)$ and exponential
$\exp^F(X)$, characterized by the three properties
\begin{align*}
\log^F(X)&\equiv X\mod{(X^2)},\\
\exp^F(\log^F(X))&=X,\\
F(X,Y)&=\exp^F(\log^F(X)+\log^F(Y)).
\end{align*}
The $[-1]$-series $[-1]_F(X)\in R[[X]]$ is characterized by 
the identity
\[
F(X,[-1]_F(X))=0,
\]
and also satisfies
\[
[-1]_F(X)=\exp^F(-\log^F(X))\equiv -X\mod{(X^2)}.
\]
We wish to understand the relationship between the two series
$S^F(X)=X+[-1]_F(X)$ and $P^F(X)=-X[-1]_F(X)$, both of which lie
in the ring $R[[X]]\subset R\otimes\Q[[X]]$. Consider the ring
automorphism of $R[[X]]$ given by
\[
\alpha\:f(X)\longleftrightarrow f([-1]_FX).
\]
By induction on $n\geq0$, the fixed subring of the quotient ring
$R_n=R[[X]]/(X^{2n+1})$ is easily seen to be
\[
R_n^\alpha=R[[P^F(X)]]/(X^{2n+1}),
\]
where the inductive step uses the fact that
\[
\alpha(X^{2n+1})\equiv-X^{2n+1}\mod{(P^F(X)^{n+1})}.
\]
A straightforward argument using a limit in the $X$-adic topology
now identifies the fixed subring of $R[[X]]$ as
$R[[X]]^\alpha=R[[P^F(X)]]$. Thus there must be an expansion
\[
S^F(X)=\sum_{r\geq1}c^F_{2r-1}P^F(X)^r
\]
with $c^F_{2r-1}\in R_{4r-2}$. These coefficients $c^F_{2r-1}$
can be determined using Lagrange Inversion, which we next recall
in a form suitable for our purposes.

Let $R((Z))=R[[Z]][Z^{-1}]$, be the ring of finite tailed Laurent 
series in $Z$ over $R$. For $f(Z)=\dsum_{n_0\leq n}a_nZ^n\in R((Z))$,
we will write
\[
[f(Z)]_{Z^n}=a_n=\text{coefficient of $Z^n$}
\]
and let
\[
\oint f(Z)\,\d Z=\left[f(Z)\right]_{Z^{-1}}
\]
be the \emph{residue of $f(Z)$ with respect to $Z$}. Let
$f'(Z)=\ds\sum_{n_0-1\leq n}(n+1)a_{n+1}Z^n$ denote the
formal derivative of $f(Z)$ with respect to $Z$.
\begin{thm}[Lagrange Inversion]
\label{thm:LI}
Let $f(Z)=\ds\sum_{n_0\leq n}a_nZ^n\in R((Z))$ and let
$h(Z)\in R[[Z]]$ with $h(Z)\equiv Z\mod{(Z^2)}$. Then for
the expansion $f(Z)=\dsum_{n_0\leq n}c_nh(Z)^n$, we have
\[
c_n=\oint\frac{f(Z)h'(Z)\,\d Z}{h(Z)^{n+1}},
\quad\text{for $n\geq n_0$}.
\]
\end{thm}

To calculate such a residue, we can use change of variable,
integration by parts and other standard techniques from
elementary calculus. We list here some that we require, their
proofs being straightforward formal versions of results from
calculus.
\begin{prop}\label{prop:1.1}
Let $f(Z),g(Z)\in R((Z))$, and let $h(Z)\in R[[Z]]$ satisfy
$h(Z)\equiv Z\mod{(Z^2)}$. Then we have
\begin{align}
\oint f(Z)\,\d Z&=\oint f(h(Z))h'(Z)\,\d Z,
\tag*{(a): Change of Variable}\\
\oint\frac{\d f(Z)}{\d Z}\,\d Z&=0,
\tag*{(b): Exactness}\\
\oint f(z)\frac{\d g(Z)}{\d Z}\,\d Z
&=-\oint g(z)\frac{\d f(Z)}{\d Z}\,\d Z.
\tag*{(c): Integration by Parts}
\end{align}
\end{prop}

We can use Lagrange Inversion, Theorem~\ref{thm:LI}, to
calculate the coefficients $c^F_{2r-1}$. First make the
changes of variable $X=\exp^F(Z)$ and $Y=Z^2$. Then
\[
S^F\(\exp^F(Z)\)=\sum_{k\geq1}2e^F_{2k-1}Y^k,
\]
where $\exp^F(Z)=\dsum_{k\geq0}e^F_kZ^{k+1}$ for
$e^F_n\in R_{2n}\otimes\Q$. Setting $Q(Y)=P^F\(\exp^F(Z)\)$
and $\bar Q(Y)=Q(Y)/Y=1+\cdots$, we obtain
\[
c^F_{2r-1}=
\oint\frac{\(\sum_{k\geq1}2e^F_{2k-1}Y^k\)Q'(Y)\,\d Y}{Q(Y)^{r+1}}.
\]
Integrating by parts and interchanging the summation and integral
signs, we obtain
\begin{align*}
c^F_{2r-1}
&=\sum_{r\geq k\geq1}\frac{2k}{r}e^F_{2k-1}
                             \oint\frac{Y^{k-1}\,\d Y}{Q(Y)^{r}}
\\
&=\sum_{r\geq k\geq1}\frac{2k}{r}e^F_{2k-1}
                          \left[\bar Q(Y)^{-r}\right]_{Y^{r-k}}.
\end{align*}
Notice that the coefficient $c^F_{2^t-1}$ has the form
\[
c^F_{2^t-1}=2e^F_{2^t-1}+\text{decomposables in $R\Q$}.
\]
In the cases $R_*=\MU_*$ or $R_*=\BP_*$ at the prime~$2$ (equipped
with their canonical formal group laws) Milnor's criterion tells
us that $c^F_{2^t-1}$ is a polynomial generator; indeed, for the
case of $\BP_*$, this gives a complete set of polynomial generators
for $\BP_*$ over $\Z_{(2)}$.

The elements $e^F_n\in R_{2n}\otimes\Q$ may be interpreted as
formal symmetric functions in infinitely many variables $t_i$.
More precisely, we view $e^F_n$ as the $n$-th elementary symmetric
function $\sigma_n(t)=\sum t_1t_2\cdots t_n$, obtained by
symmetrising the monomial $t_1t_2\cdots t_n$. Thus we have the
formal expansion
\[
\sum_{n\geq0}e^F_nZ^n=\prod_{i}\(1+t_iZ\),
\]
from which it follows that
\[
\bar Q(Y)=\prod_{i}\(1-t_i^2Y\).
\]
Hence we deduce that
\begin{align*}
\bar Q(Y)&=\sum_{n\geq0}q^F_nY^n
\\
&=\sum_{n\geq0}(-1)^n\sigma_n(t^2)Y^n,
\end{align*}
where $\sigma_n(t^2)=\sum t_1^2t_2^2\cdots t_n^2$ is the $n$-th
elementary symmetric function in the $t_i^2$. It is easily seen
that
\[
\sigma_n(t^2)
=2e^F_{2n}-2e^F_{2n-1}e^F_{1}+\cdots+(-1)^n\(e^F_{n}\)^2.
\]
Of course we actually want to know the first few coefficients
in $\bar Q(Y)^{-n}$ for $n\geq1$, and these will be complicated
polynomials in the $\sigma_k(t^2)$ for $1\leq k\leq n$. An
alternative approach is to use the total symmetric functions
rather than the elementary functions.

Recall that the $n$-th total symmetric function $\tau_n(t)$ in
the variables $t_i$ is obtained by the summing up all the monomials
of the form $t_1^{r_1}t_2^{r_ 2}\cdots t_n^{r_n}$ where
$0\leq r_1\leq r_2\leq\cdots\leq r_n$ and $r_1+r_2+\cdots r_n=n$,
and then symmetrising over all $t_i$'s. The generating function
for the $\tau_n(t)$ is
\[
\sum_{0\leq n}\tau_n(t)Z^n=\prod_{i}\(1-t_iZ\)^{-1}.
\]
In the context above where we set $e^F_n=\sigma_n(t)$, thus we
have
\[
\bar Q(Y)^{-1}=\sum_{0\leq n}\tau_n(t^2)Y^n
\]
and we set $h^F_{n}=\tau_n(t^2)$ which is an integer polynomial
in the elements $q^F_k$. Thus to evaluate $c^F_{2r-1}$ we are
reduced to a calculation with polynomials in the~$h^F_{n}$.

We end this section with another technical result required later.
The proof is a well known application of Newton iteration.
\begin{thm}[Hensel's Lemma]
\label{thm:HL}
Let $R$ be a commutative, unital ring which is complete with respect
to the $I$-adic topology for some ideal $I\ideal R$ and let
$f(Z)\in R[[Z]]$ be a power series such that $f(z_0)\equiv 0\mod{I}$
for some $z_0\in R$ and the formal derivative $f'(z_0)$ is a unit
in $R$. Then the sequence $z_n\in R$ given by
\[
z_{n+1}=z_n-\frac{f(z_n)}{f'(z_0)}
\]
converges $I$-adically to a limit $z$ satisfying $f(z)=0$.
\end{thm}

\section{A universal formal group law for $\MSp[1/2]$}
\label{sec:2}

The main results of this section are the following. First we
follow Quillen in using use formal group theory to prove
\begin{thm}\label{thm:2.3}
There is an idempotent map of ring spectra
$\epsilon_2\:\MU[1/2]\lra\MU[1/2]$ for with image ring spectrum
$\evMU$ and a canonical orientation $\MSp\lra\evMU$ inducing
an equivalence of ring spectra $\MSp[1/2]\lra\evMU$.
\end{thm}

Next we deduce a topological result which in effect shows
that $\MU[1/2]$ is a `Hopf algebra over $\MSp$' in the sense
that the natural product and diagonal maps of $\Sp/\U$ give
$\MSp[1/2]_*\Sp/\U$ the structure of a Hopf algebra over
$\MSp[1/2]_*$.
\begin{thm}\label{thm:2.5}
There is an equivalence of ring spectra
\[
\MSp[1/2]\Smash\Sigma^\infty(\Sp/\U_+)\lra\MU[1/2].
\]
\end{thm}

Finally we obtain
\begin{thm}\label{thm:2.6}
There is an equivalence of ring spectra
\[
\MU[1/2]\homeq\bigvee_{J}\Sigma^{2|J|}\MSp[1/2]
\]
where $J$ ranges over all non-decreasing finite odd sequences
\[
J=(1\leq2j_1-1\leq2j_2-1\leq\cdots\leq2j_\ell-1),
\]
including the empty sequence $()$, and
$|J|=\dsum_{1\leq t\leq\ell}(2j_t-1)$; moreover, these summands
are indexed by a polynomial algebra $\Z[1/2][u_J:J\neq()]$ with
$|u_J|=2|J|$ and have multiplication compatible with this product.
\end{thm}

We begin by constructing an idempotent map of ring spectra
$\epsilon_2\:\MU[1/2]\lra\MU[1/2]$ for which the cohomology
theory associated to the image is representable by a ring
spectrum equivalent to $\MSp[1/2]$. Our construction is
modelled on that of Quillen's idempotent for $\BP$, as
described by Adams~\cite{JFASHGH}.

Let $x=x^\MU\in\MU^2(\CP^\infty)\subset\MU[1/2]^2(\CP^\infty)$
denote the canonical orientation class of~\cite{JFASHGH} which
arises geometrically as the homotopy class of a map
$\CP^\infty\xrightarrow{\homeq}\MU(1)\lra\Sigma^2\MU$.
Let $[1/2]_\MU(X)\in\MU[1/2]_*[[X]]$ denote the series
characterized by
\[
[2]_\MU([1/2]_\MU(X))=X.
\]
The series
\[
\phi(X)=[1/2]_\MU\(F^\MU(X,[-1]_\MU(-X))\)\in\MU[1/2]_*[[X]]
\]
may be interpreted as a strict isomorphism $\phi\:F'\iso F$,
where $F'$ is the formal group law defined by
\[
F'(X,Y)=\phi\(F\(\phi^{-1}(X),\phi^{-1}(Y)\)\),
\]
We choose a new orientation $y\in\MU[1/2]^2(\CP^\infty)$ by
requiring it to satisfy
\[
x=[1/2]_\MU\(F^\MU(y,[-1]_\MU(-y))\)
\]
and noting that in $\MU[1/2]^*(\CP^\infty)$ we have
\[
y\equiv x\mod{(x^2)}.
\]
Then $y$ is an orientation whose associated formal group
law is $F'$.
\begin{defn}\label{defn:odd}
A formal group law $G$ for which $[-1]_G(X)=-X$ is said to
be \emph{odd}.
\end{defn}

The logarithm of $F'$ is
\begin{align*}
\log^{F'}(X)&=\frac{1}{2}\(\log^\MU(X)-\log^\MU(-X)\)
\\
            &=\sum_{n\geq0}m_{2n}X^{2n+1},
\end{align*}
where $\log^\MU(X)=\sum_{n\geq0}m_nX^{n+1}\in\MU_*\otimes\Q[[X]]$.
Clearly $\log^{F'}(X)$ is an odd function of $X$ and a straightforward
calculation shows that $[-1]_{F'}(X)=-X$, hence $F'$ is odd.

The universality of $\MU_*$ implies that there is a unique homomorphism
$\MU_*\lra\MU[1/2]_*$ which classifies $F'$ and gives rise to a
homomorphism $\epsilon_2\:\MU[1/2]_*\lra\MU[1/2]_*$. The latter
is an idempotent, since it extends to a visibly idempotent
homomorphism on $\MU_*\otimes\Q=\Q[m_n:n\geq1]$, given by
\[
m_n\longmapsto
\begin{cases}
0    &\text{if $n$ odd},\\
m_n  &\text{if $n$ even}.
\end{cases}
\]
Let $\evMU_*=\im\epsilon_2\subset\MU[1/2]_*$ and let $\evMU^*(\ )$
denote the associated complex oriented multiplicative cohomology
theory, which is a summand of complex cobordism with $2$ inverted,
$\MU[1/2]^*(\ )$.

Explicit polynomial generators for $\evMU_*$ can be given using
the method of I.~Kozma~\cite{Kozma}. For each prime~$\ell$, let
$\f_\ell$ and $\v_\ell$ denote the \emph{Frobenius} and
\emph{verschiebung} operators with respect to the formal group
law~$F$. Thus, using ${\sum}^F$ to denote formal group summation,
we have
\begin{align*}
\f_\ell(X)&={\sum_{\zeta^\ell=1}}^F(\zeta X^{1/\ell}),
\\
\v_\ell(X)&=X^{\ell},
\\
\intertext{and}
\v_\ell\f_\ell(X)&={\sum_{\zeta^\ell=1}}^F(\zeta X^{\ell})
\\
                 &={\sum_{k\geq1}}^F(\T_{\ell,k}X^{k\ell})
\end{align*}
for some $\T_{\ell,k}\in\MU_{2k\ell-2}$. In the ring
$\MU_*\otimes\Q$ these coefficients satisfy the identities
\[
\T_{\ell,k}=
\ell m_{k\ell-1}-\sum_{1<s\mid k}m_{s-1}\T_{\ell,k/s}^{\,s}.
\]
Using the well known criterion of Milnor (see~\cite[theorem, page~128]{Stong},
also~\cite[part~II, Theorem~7.9]{JFASHGH}) we obtain the following.
If $k=\ell^d$, then $\T_{\ell,\ell^d}$ is a polynomial generator
of $\MU_*$. If $n$ has two distinct prime factors $\ell_1$ and $\ell_2$,
then some integral linear combination of $\T_{\ell_1,n/\ell_1}$ and
$\T_{\ell_2,n/\ell_2}$ will be a polynomial generator. We will assume
that a choice of such an element $\X_{n-1}\in\MU_{2n-2}$ has been made
for each $n$, with $\X_{\ell^{d+1}-1}=\T_{\ell,\ell^d}$ for every prime
$\ell$. Then by induction we have
\[
\epsilon_2(\T_{\ell,k})=
\begin{cases}
0&\text{if $\ell k$ is even},\\
\T_{\ell,k}&\text{if $\ell k$ is odd}.
\end{cases}
\]
\begin{thm}\label{thm:2.1}
The elements $\X_{2m}$ form a set of polynomial generators
for $\evMU_*$ as a $\Z[1/2]$-algebra.
\end{thm}

Suppose $(R,G)$ is a pair consisting of a $\Z[1/2]$-algebra
$R$ together with an odd formal group law $G$ on $R$ classified
by a ring homomorphism $\psi\:\MU[1/2]_*\lra R$. The pushed
forward power series
\[
\psi_*\phi(X)=[1/2]_G\(G(X,[-1]_G(-X))\)\in R[[X]]
\]
has the associated formal group law
\[
G'(X,Y)=\psi_*\phi\(G\(\psi_*\phi^{-1}(X),\psi_*\phi^{-1}(Y)\)\),
\]
classified by the homomorphism $\psi\o\epsilon_2$. But since $G$
is odd, $\psi_*\phi(X)=X$, hence $G'=G$ and so $\psi\o\epsilon_2=\psi$.
Hence $\psi$ has a unique factorization
\[
\psi\:\MU[1/2]_*\xrightarrow{\epsilon_2}\evMU_*\lra R.
\]
We may now easily deduce
\begin{thm}
\label{thm:2.2}
The pair $(\evMU_*,F')$ is universal for pairs $(R,G)$ consisting
of a $\Z[1/2]$-algebra $R$ together with an odd formal group
law $G$ on $R$.
\end{thm}

The theory $\evMU^*(\ )$ possesses canonical orientations for
symplectic bundles, which can be regarded as complex bundles.
In particular, the canonical quaternionic line bundle
$\xi_1\lra\HP^\infty$ has an orientation $w\in\evMU^4(\HP^\infty)$,
and under the pullback induced from the canonical map
$\CP^\infty\lra\HP^\infty$, this maps to
\[
y(-[-1]_{F'}y)=y^2\in\evMU^4(\CP^\infty).
\]
For any commutative ring $R$, the image of the homomorphism
$x_*\:H_*(\CP^\infty;R)\lra H_*(\MU;R)$ induced by the orientation~$x$
contains a set of algebra generators for $H_*(\MU;R)$. Similarly,
the map
\[
y\:\CP^\infty\xrightarrow{\homeq}\MU(1)\xrightarrow{x}\Sigma^2\MU[1/2]
\lra\Sigma^2\,\evMU[1/2]\lra\Sigma^2\MU[1/2]
\]
induces a map in homology which has generators in its image.
Now it follows that the composite
\[
\CP^\infty\xrightarrow{\homeq}\MU(1)
\xrightarrow{y}\Sigma^2\;\evMU[1/2]
\lra\Sigma^2\MU[1/2]
\]
also provides algebra generators for
$H_*(\evMU;\Z[1/2])\subset H_*(\MU;\Z[1/2])$. Then it follows that
the map $w\:\HP^\infty\homeq\MSp(1)\lra\Sigma^4\MSp\lra\evMU$ induces
a map in homology whose image contains polynomial generators for
$H_*(\evMU;\Z[1/2])$, since the natural map $\CP^\infty\lra\HP^\infty$
induces a surjection in ordinary homology. Finally, we see that the
orientation map $\MSp\lra\evMU$ induces an isomorphism in homology,
since $H_*(\MSp;\Z[1/2])$ is generated by the image of the natural
homomorphism $H_*(\HP^\infty;\Z[1/2])\lra H_*(\MSp;\Z[1/2])$ and we
know that $H_*(\MSp;\Z[1/2])$ maps surjectively to $H_*(\evMU;\Z[1/2])$.

Recall the existence of the fibrations
\[
\Sp(n)/\U(n)\lra\B\U(n)\lra\B\Sp(n),
\]
induced from the `quaternionification map' $\U(n)\lra\Sp(n)$.
For each $m$, there is a map
\[
\B\U(m)\times\Sp(n)/\U(n)\lra\B\U(m+n)
\]
constructed using the external Whitney sum map. Hence there
is a family of maps compatible under increasing values of $n$
and $k$,
\[
\B\Sp(k)\times\Sp(n)/\U(n)\lra\B\U(2k)\times\Sp(n)/\U(n)\lra\B\U(2k+n).
\]
A straightforward calculation now shows that on passing to
the limit with respect to $n$ and $k$, there is a map of H-spaces
$\B\Sp\times\Sp/\U\lra\B\U$ which induces an isomorphism of Hopf
algebras over $\Z[1/2]$,
\[
H_*(\B\Sp\times\Sp/\U;\Z[1/2])\lra H_*(\B\U;\Z[1/2]).
\]

We now pass to Thom spectra. Let $\chi\lra\Sp/\U$ denote the
pullback of the universal (virtual) bundle $\zeta\lra\B\U$. The
Thom complex $\M\chi$ of $\zeta$ is a ring spectrum and possesses
a canonical orientation $\M\chi\lra\MU$. Smashing with $\MSp$ we
obtain a morphism of ring spectra $\MSp\Smash\M\chi\lra\MU$ and
a commutative diagram of $\Z[1/2]$-algebras,
\[
\begin{CD}
H_*(\B\Sp\times\Sp/\U;\Z[1/2])@>\iso>>H_*(\B\U;\Z[1/2]) \\
@V\iso VV     @V\iso VV \\
H_*(\MSp\Smash\M\chi;\Z[1/2])@>>> H_*(\MU;\Z[1/2])
\end{CD}
\]
in which the vertical maps are Thom isomorphisms in $H_*(\ ;\Z[1/2])$,
and hence the bottom row is an isomorphism.

We would like to replace $\MSp\Smash\M\chi$ up to homotopy
by $\MSp\Smash\Sigma^\infty(\Sp/\U_+)$, at least after inverting 2.
We are able to do this using the following result from~\cite{JNMM}
provides a geometric realization of the Thom isomorphism.
\begin{lem}
\label{lem:2.4}
Let $\xi\lra B$ be a virtual bundle, orientable in the cohomology
theory represented by the ring spectrum $E$, and let
$u\:\M\xi\lra E$ be an orientation. Then there is an equivalence
\[
E\Smash\M\xi\homeq E\Smash\Sigma^\infty(B_+).
\]
\end{lem}


The equivalence of Theorem \ref{thm:2.5} is the composite
\[
E\Smash\M\xi\xrightarrow{1\Smash\delta\phantom{\Smash1}}
E\Smash\M\xi\Smash\Sigma^\infty(B_+)
\xrightarrow{1\Smash u\Smash1}E\Smash E\Smash\Sigma^\infty(B_+)
\xrightarrow{\phantom{1\Smash}\mu\Smash1}E\Smash\Sigma^\infty(B_+),
\]
where $u\:\M\xi\lra E$ is the chosen orientation for $\xi$,
$\delta\:\M\xi\lra\M\xi\Smash\Sigma^\infty(B_+)$ is the
external diagonal used in defining the Thom isomorphism,
and $\mu\:E\Smash E\lra E$ is the multiplication map.

Theorem \ref{thm:2.6} now follows by a standard argument
from the fact that
\[
\MSp[1/2]_*(\Sigma^\infty(\Sp/\U_+))
\iso\evMU_*(\Sigma^\infty(\Sp/\U_+))
\]
is a polynomial algebra over the ring $\MSp[1/2]_*\iso\evMU_*$,
hence is a free module.

\section{Categories of odd formal group laws}
\label{sec:3}

In this section, let $R=R_*$ be a commutative unital graded
$\Z[1/2]$-algebra. We will also denote the category of commutative,
1-dimensional formal group laws over $R$ by $\FGL(R)$, and the full
category of odd formal group laws by $\evFGL(R)$; here the morphisms
are strict isomorphisms. Thus a morphism $F_1\xrightarrow{\phi}F_2$
between two formal group laws is a power series~$\phi(X)\in R[[X]]$
such that
\begin{align*}
\phi(X)&\equiv X\mod~{(X)^2}
\\
\intertext{and}
\phi(F_1(X,Y))&=F_2(\phi(X),\phi(Y)).
\end{align*}
These small categories are groupoids, and are representable covariant
functors of the algebra~$R$. There are natural isomorphisms
\begin{align}
\FGL(R)&\iso\Alg_{\Z[1/2]}(\MU[1/2]_*\MU[1/2],R),
\label{eqn:3.1}\\
\evFGL(R)&\iso\Alg_{\Z[1/2]}(\evMU_*\evMU,R),
\label{eqn:3.2}
\end{align}
where we define
\[
\evMU_*\evMU=
\evMU_*\oTimes{\MU[1/2]_*}\MU[1/2]_*\MU[1/2]\oTimes{\MU[1/2]_*}\evMU_*,
\]
making use of the fact that $\MU[1/2]_*\MU[1/2]$ is a bimodule over
$\MU[1/2]_*$ and the idempotent $\epsilon_2$ provides $\evMU_*$ with
the structure of a module over $\MU[1/2]_*$. The topological structure
of $\MU_*\MU$ and $\evMU_*\evMU$ includes that of `bilateral Hopf
algebras' or `Hopf algebroids', and the associated structure maps
yield (natural with respect to the algebra $R$) groupoid structures
on the right hand sides of Equations \eqref{eqn:3.1} and \eqref{eqn:3.2},
and these isomorphisms are isomorphisms of groupoids. It also follows
from standard results in this area~\cite{HRMDCR,DCRBk} that the
idempotent $\epsilon_2$ induces an equivalence of groupoids
\[
\boldsymbol\epsilon_{\mathbf2}\:\FGL(R)\lra\evFGL(R),
\]
which is functorial in $R$. The fixed objects are precisely the odd
formal group laws, while the
fixed morphisms are the strict isomorphisms between odd formal
group laws which are odd power series over $R$. The full subcategory
of all such objects and morphisms is the image of
$\boldsymbol\epsilon_{\mathbf2}$.

V.~M.~Buchstaber~\cite{VMB} introduced a different idempotent
$e_2$ on $\MU[1/2]$ with associated idempotent equivalence
\[
\mathbf e_{\mathbf2}\:\FGL(R)\lra\evFGL(R).
\]
This is defined as follows. Associated to a formal group law
$F$ on $R$ there is a strict isomorphism
\[
\theta_F\:F\lra\mathbf e_{\mathbf2}(F)
\]
given by the series
\[
\theta_F(X)=X\sqrt{\dfrac{-[-1]_F(X)}{X}},
\]
where for any series $h(X)\in R[[X]]$ satisfying $h(X)\equiv1\mod{(X)}$,
we define
\[
\sqrt{h(X)}\equiv1\mod{(X)}
\]
to be the unique square root of $h(X)$ with this property (this of course
depends crucially on the fact that $R$ is an algebra over $\Z[1/2]$);
this series can also be determined using the formal binomial expansion
for $h(X)^{1/2}$. It is now straightforward to verify the following result.
\begin{prop}
\label{prop:3.3}
The idempotents $\epsilon_2$ and $e_2$ satisfy the following.
\begin{enumerate}
\item[a)]
For any formal group law over $R$, $\mathbf e_{\mathbf 2}(F)$ is odd.
\item[b)]
For any strict isomorphism $\phi\:F_1\lra F_2$,
$\mathbf e_{\mathbf 2}(\phi)$ is odd.
\item[c)]
There are equations
\begin{align*}
\mathbf  e_{\mathbf 2}\boldsymbol\epsilon_{\mathbf 2}
&=\mathbf e_{\mathbf 2},
\\
\boldsymbol\epsilon_{\mathbf 2}\mathbf e_{\mathbf 2}
&=\boldsymbol\epsilon_{\mathbf 2}.
\end{align*}
\item[d)]
The restrictions of the idempotents $\epsilon_2$ and $e_2$
induce inverse isomorphisms of rings
\[
\im e_2
\begin{matrix}
\xrightarrow{\ds\epsilon_2}\\[-7pt]
\xleftarrow[\ds e_2]{}
\end{matrix}
\im\epsilon_2.
\]
\end{enumerate}
\end{prop}

It is worth noting that Buchstaber's idempotent
$e_2\:\MU[1/2]_*\lra\MU[1/2]_*$ fixes the image of the canonical
map $\MSp_*\lra\MU[1/2]_*$. Indeed there is a diagram of ring
spectra
\[
\begin{CD}
\MSp[1/2]@>>>\MU[1/2] \\
@V\id VV  @V{e_2}VV \\
\MSp[1/2]@>>>\MU[1/2]
\end{CD}
\]
from which we can deduce that the image of the idempotent
$e_2\:\MU[1/2]^*(\ )\lra\MU[1/2]^*(\ )$ is a cohomology theory
which agrees with the image of the associated natural transformation
of cohomology theories $\MSp[1/2]^*(\ )\lra\MU[1/2]^*(\ )$. However,
the same is \emph{not} true for our idempotent $\epsilon_2$.

From~\cite{Swit} we obtain the following facts about the spectrum
$\MSp$, which do not use any detailed knowledge of the structure
of the still mysterious ceofficient ring $\MSp_*$. The object
\[
\MSp_*\MSp=\MSp_*(\MSp)=\MSp_*[Q_n:n\geq1]
\]
is endowed with the structure of a Hopf algebroid, where
the elements $Q_n\in\MSp_{4n}\MSp$ may be chosen so that
the coproduct on the $Q_n$ is given by the composition law
\begin{align*}
Q(T)&\longmapsto(1\otimes Q)(Q\otimes1(T)),
\\
\intertext{where}
Q(T)&=\sum_{n\geq0}Q_nT^{n+1},
\\
(1\otimes Q)(T)&=\sum_{n\geq0}1\otimes Q_nT^{n+1},
\\
(Q\otimes1)(T)&=\sum_{n\geq0}Q_n\otimes1\,T^{n+1}.
\end{align*}
Furthermore, the natural ring homomorphism $\MSp_*\MSp\lra\MU_*\MU$
is a morphism of Hopf algebroids over $\Z$. After inverting 2 and
applying the idempotent $\epsilon_2$ on the factors of $\MU[1/2]$,
we obtain the following result.
\begin{thm}\label{thm:3.4}
The evident natural homomorphism $\MSp[1/2]_*\MSp[1/2]\lra\evMU_*\evMU$
is an isomorphism of Hopf algebroids over $\Z[1/2]$.
\end{thm}

Notice that the natural orientation $p^\MSp\in\MSp^4(\HP^\infty)$
maps to an orientation $\pev\in\evMU^4(\HP^\infty)$, and on pulling
back to $\CP^\infty$, becomes the square of the complex orientation
$y\in\evMU^2(\CP^\infty)$.

The morphism of Hopf algebroids $\MU[1/2]_*\MU[1/2]\lra\evMU_*\evMU$
induced by the idempotent~$\epsilon_2$ is an equivalence, in the
sense that there is a commutative diagram of groupoids which is
natural in $R$
\[
\begin{CD}
\Alg_{\Z[1/2]}(\MU[1/2]_*\MU[1/2],R)@>>>\Alg_{\Z[1/2]}(\evMU_*\evMU,R) \\
@V\iso VV                                @V\iso VV   \\
\FGL(R)@>>>\evFGL(R)
\end{CD}
\]
and in which the rows are equivalences of groupoids induced by
$\epsilon_2$. For example, this implies that the associated Ext
groups arising as the Adams-Novikov $\mathrm{E}_2$-terms for
the homology theories defined by $\MU[1/2]$ and $\MSp[1/2]$ are
naturally isomorphic.

There is a well known decomposition of $\Z$-algebras,
\[
\MU_*\MU=\MU_*\oTimes{\Z}\mathrm S_*,
\]
where $\mathrm S_*=\Z[B_n:n\geq1]$ with $B_n$ the usual generator
of $\MU_*\MU$ over $\MU_*$. This splitting is such that the subalgebra
$\MU_*$ is invariant with respect to the coaction map over $\mathrm S_*$.
Dually, we have a semi-tensor decomposition of $\Z$-algebras,
\[
\MU^*\MU=\MU_*\oTimes{\Z}\mathrm S^*,
\]
where $\mathrm S^*=\Hom_\Z(\mathrm S_*,\Z)$. Here the subalgebra
$\MU_*$ is invariant under the action of $\mathrm S^*$. The idempotent
${\epsilon_2}_*\epsilon_2$, obtained by applying $\epsilon_2$ to
both the factors of $\MU$ in $\MU_*\MU$, has image
\[
{\epsilon_2}_*\epsilon_2(\mathrm S_*)=\Z[1/2][B_{2n}:n\geq1]
=\evS_*
\]
when restricted to $\mathrm S_*[1/2]$.

These decompositions are compatible with the action of the
idempotent $\epsilon_2$, and there is an isomorphism
\[
\MSp[1/2]^*\MSp[1/2]\iso\evMU_*\otimes\evS^*,
\]
where $\evS^*=\Hom_{\Z[1/2]}(\evS_*,\Z[1/2])$. There is a well
known semi-tensor product decomposition
\[
\MSp^*\MSp\iso\MSp_*\oTimes{\Z}\mathrm Q^*,
\]
where is dual to $\mathrm Q_*=\Z[Q_n:n\geq1]$; similarly, the
$\Z[1/2]$-Hopf algebra $\evS^*$ is isomorphic to the dual of
$\Z[1/2][Q_n:n\geq1]\subset\MSp[1/2]_*\MSp[1/2]$.

The work of \cite{BS,Mor} suggests that we should be able to
relate the algebra $\evS^*$ to the Virasoro algebra. In~\cite{Mor},
the operations $s_{e_n}\in\MU^{2n}\MU$ of~\cite{JFASHGH} are shown
to realize the action of the operators $v_k=z^{k+1}\d/\d z$ in the
rational Lie algebra $\mathrm V^+$ with basis $\{v_k:k\geq1\}$; it
follows that the operations $\epsilon_2\o s_{e_n}\o\epsilon_2$
realize the action of the even half of $\mathrm V^{+}$, namely the
subalgebra $\mathrm V^{2+}$ with basis $\{v_{2k}:k\geq1\}$. In this
way, we can interpret the Hopf algebra $\Q\otimes\mathrm Q^*$ as the
universal enveloping algebra of $\mathrm V^{2+}$. A perhaps more
interesting point involves the action of this even part of the
Virasoro algebra on the generators of $\MSp[1/2]_*(\Sp/\U_+)$
generating the augmentation ideal.

\section{Some Witt vector-like Hopf algebras}\label{sec:4}

The constructions in this section are reminiscent of, and influenced
by, the algebra of Witt vectors associated to a formal group law as
described in~\cite{Haz}. However, they differ in ways that appear
novel. Throughout the section, let $\k=\k_*$ be any graded commutative
unital ring.

The algebra $H_*=\k[b_n:n\geq1]$ with $b_n\in H_{2n}$
and $b_0=1$ admits the cocommutative coproduct
\[
\psi(b_n)=\sum_{0\leq k\leq n}b_k\otimes b_{n-k}
\]
and antipode $\chi$ for which
\[
\sum_{0\leq k\leq n}\chi(b_k)b_{n-k}=0.
\]
In terms of generating functions we have for the series
$b(T)=\dsum_{0\leq n}b_nT^n$,
\begin{align*}
\psi b(T)=\sum_{0\leq k\leq n}\psi(b_n)T^n&=b(T)\otimes b(T),
\\
\intertext{and}
\chi b(T)=\sum_{0\leq k\leq n}\chi(b_n)T^n&=b(T)^{-1}.
\end{align*}
This Hopf algebra represents the group scheme of sequences
$\mathbf{W}$ on graded $\k$-algebras for which
\[
\mathbf{W}(R)=\Alg_{\k}(H_*,R),
\]
where $\phi\in\Alg_{\k}(H_*,R)$ is identified with the sequence
$\{\phi(b_n)\}_{0\leq n}$; `addition' is given on sequences by
\[
\{c_n\}*\{d_n\}=\left\{\sum_{0\leq k\leq n}c_kd_{n-k}\right\}.
\]
Writing $c=\{c_n\}$ and $d=\{d_n\}$, we will also write
\[
c*d=\{(c*d)_n\}=\{c_n\}*\{d_n\}.
\]

Now suppose that $1/2\in\k$. There is an endomorphism of $\mathbf W$
determined by
\[
w\longmapsto(1/2)w
\]
which is induced by the Hopf algebra homomorphism satisfying
\[
b(T)\longmapsto b(T)^{1/2},
\]
which can be calculated using the formal binomial expansion. There
is also an involution $\tau$ given by
\[
\tau\.b(T)=b^\tau(T)=b(-T).
\]
There are two idempotent endomorphisms on $\mathbf W$ induced by
the Hopf algebra endomorphisms
\begin{align*}
b(T)&\longmapsto b^+(T)=b(T)^{1/2}b(-T)^{1/2}, \\
b(T)&\longmapsto b^-(T)=b(T)^{1/2}b(-T)^{-1/2}.
\end{align*}
A straightforward calculation shows that
\begin{align*}
b^+_{2n}&\equiv b_{2n}\mod{(\text{decomposables})}, \\
b^-_{2n-1}&\equiv b_{2n-1}\mod{(\text{decomposables})},
\end{align*}
hence these elements form as set of polynomial generators
for $H_*$.
\begin{remark}\label{rem:b+}
Although $b^+_{2n-1}=0$, the elements $b^-_{2n}\in H^-_*$
need not be zero since $b^-(T)$ is not an odd series.
\end{remark}

We will denote the images of these by $\mathbf W^+$ and $\mathbf W^- $,
where $\mathbf W^+(R)=\mathbf W(R)^\tau$. Denoting the corresponding
representing Hopf algebras by $H^+_*$ and $H^-_*$, we have
\[
H^+_*=\k[b^+_{2n}:n\geq1], \quad H^-_*=\k[b^-_{2n-1}:n\geq1].
\]
The following result is now immediate.
\begin{thm}
\label{thm:4.1}
There is a decomposition of group schemes
\[
\mathbf W=\mathbf W^+\times\mathbf W^-,
\]
or equivalently of Hopf algebras
\[
H_*=H^+_*\oTimes{\k}H^-_*.
\]
\end{thm}
Let $c=\{c_n\}_{0\leq n}$ be a sequence in the $\k$-algebra $R$
with $c_0=1$. The power series $c(X)X=\dsum_{0\leq n}c_nX^{n+1}$
has a composition inverse $\tilde c(X)X=\dsum_{0\leq n}\tilde c_nX^{n+1}$;
the associated sequence $\tilde c=\{\tilde c_n\}_{0\leq n}$ is called
the \emph{reverted series} or the \emph{reversion} of $c$.
When $R$ has no $\Z$-torsion, we may determine $\tilde c_n$ by using
Lagrange Inversion to obtain the well known formula
\[
\tilde c_n=
\frac{1}{n+1}
\left[
c(T)^{-n-1}
\right]_{T^n},
\]
i.e., the coefficient of $T^n$ in $\dfrac{1}{n+1}c(T)^{-n-1}$.
Notice that $\tilde{\tilde c}=c$. Following Adams~\cite{JFASHGH}
and working in the algebra $H_*$, we denote the reversion of
$b=\{b_n\}$ by $m=\{m_n\}=\{\tilde b_n\}$. Since
\[
m_n\equiv -b_n\mod{(\text{decomposables})},
\]
we have $H_*=\k[m_n:n\geq1]$. We can interpret this as defining
yet another group scheme $\tilde{\mathbf W}$, for which
\[
\tilde{\mathbf W}(R)=\Alg_{\k}(H_*,R)
\]
is identified with the set of all sequences in $R$, but this time
$\phi\in\Alg_{\k}(H_*,R)$ corresponds to $\{\phi(m_n)\}$ rather
than $\{\phi(b_n)\}$. The addition law here is given by
\begin{align*}
c\diamond d&=\tilde{\tilde c*\tilde d}
\\
\intertext{or equivalently,}
\{c_n\}\diamond \{d_n\}&=\left\{\tilde{(\tilde c*\tilde d)}_n\right\}.
\end{align*}
The coproduct in $H_*$ gives rise to this group structure. Of course,
there is an isomorphism of group schemes $\mathbf W\iso\tilde{\mathbf W}$
for which on $R$,
\[
\{c_n\}\longleftrightarrow\{\tilde c_n\}.
\]
This is induced by the algebra automorphism of $H_*$ given by
\[
b_n\longleftrightarrow m_n.
\]

Now let $F$ be a formal group law over the $\k$-algebra $R$.
For simplicity, we assume that $F$ is \emph{odd}, however it
is possible to rework our discussion using an arbitrary formal
group law and replacing certain occurrences of $-T$ by the
$[-1]_F$-series.

Given a sequence $\{c_n\}_{0\leq n}$ in $R$, we can consider
the series $\dsum_{0\leq n}c_nX^{n+1}$. A standard calculation
in formal group theory shows that there are unique elements
$c^F_n$ in $R$ for which
\[
\sum_{0\leq n}c_nX^{n+1}=\sum^F_{0\leq n}c^F_nX^{n+1}.
\]
Moreover, each $c^F_n-c_n$ is a polynomial over $\k$ in the
coefficients of $F$ together with the $c^F_i$ and $c_i$ for
$i<n$. Given two such sequences $c=\{c_n\}$ and $d=\{d_n\}$,
we have
\[
(c\diamond d)(X)=
\sum^F_{0\leq n}(c^F\Dia{F}d^F)_nX^{n+1}
\]
where $(c\Dia{F}d)^F_n-c^F_n-d^F_n$ is a polynomial over $\k$ in
the $c^F_i$ and $d^F_i$ for $i<n$ together with the coefficients
of $F$. This gives a group scheme $\mathbf W^F$ (depending on $F$)
for which the underlying set of $\mathbf W^F(R)$ again consists
of sequences in $R$, but this time we have the addition law
\[
\{c_n\}\Dia{F}\{d_n\}=\{(c\Dia{F}d)_n\}.
\]
As a representing object for $\mathbf W^F$ we have
\[
H^F_*=\k[m^F_n:n\geq1],
\]
where we view the generators $m^F_n$ as the coefficients of
the logarithm of the formal group law $F$. Thus
\[
\mathbf W^F(R)\iso\Alg_{\k}(H^F_*,R),
\]
where we identify $\phi\in\Alg_{\k}(H^F_*,R)$ with the sequence
$\{\phi(m^F_n)\}$. $H^F_*$ admits a coproduct $\psi_F$ and antipode
$\chi^F$ giving rise to the above group structure.
\begin{thm}\label{thm:4.2}
There is an isomorphism of group schemes
$\tilde{\mathbf W}\iso\mathbf W^F$, or equivalently of
Hopf algebras $H_*\iso H^F_*$.
\end{thm}
\begin{proof}
We identify a series $\{c_n\}\in\tilde{\mathbf W}(R)$ with a power
series
\[
\sum_{1\leq n}c_nX^n=\sum^F_{1\leq n}(c^F_nX^n)
\]
and hence with an element of $\{c^F_n\}$ of $\mathbf W^F(R)$.
Moreover, the group structures match up by definition of $(c\Dia{F}d)$
\end{proof}

We may also use the decomposition $\mathbf W\iso\mathbf W^+\times\mathbf W^-$
to induce a similar decomposition of $\mathbf W^F$.
\begin{thm}\label{thm:4.3}
There is an decomposition of group schemes
\[
\mathbf W^F=\mathbf W^{F+}\times\mathbf W^{F-},
\]
or equivalently of Hopf algebras
\[
H^F_*=H^{F+}_*\oTimes{\k}H^{F-}_*.
\]
\end{thm}

It is worth remarking that the Hopf algebras $H^{F+}_*$ and $H^{F-}_*$
are polynomial,
\[
H^{F+}_*=\k[m^{F+}_{2n}:n\geq1],
\quad
H^{F-}_*=\k[m^{F-}_{2n-1}:n\geq1],
\]
but that an element of $c\in\mathbf W^{F+}(R)$ has the form $\{c_n\}$
in which not all of the $c_{2n-1}$ need be~$0$. Similarly, for
$c\in\mathbf W^{F-}(R)$, the terms $c_{2n}$ need not vanish.

\section{Applications to some Hopf algebras from algebraic topology}
\label{sec:5}

If $\k=E_*$ is the coefficient ring of a complex oriented
cohomology theory, then $E_*(\BU_+)\iso E_*[b_n:n\geq1]$,
where we may identify $b_n$ with the standard generator
$\beta^E_n$ of~\cite{JFASHGH}.

A particular case of interest is the universal one, $E=\MU$. This
is made more interesting by the existence of a map of ring spectra
\[
\Delta\:\MU\lra\MU\Smash\BU_+=\MU\Smash\Sigma^\infty(\BU_+),
\]
which is the `external Thom diagonal' (see the discussion after
Theorem~\ref{thm:2.5}). This gives rise to a multiplicative
cohomology operation
\[
\bar\Delta\:\MU^*(\ )\lra(\MU\Smash\BU_+)^*(\ )
\iso\MU_*(\BU_+)\oTimes{\MU_*}\MU^*(\ )
\]
whose effect on the orientation class $x\in\MU^2(\CPi)$ is given
by
\[
\bar\Delta(x)=b(x)x=\sum_{n\geq0}b_nx^{n+1}.
\]

By Quillen's result~\cite{Qu,JFASHGH} identifying $\MU_*$ with
Lazard's ring, we can view the $\MU_*$-algebra
\[
\MU_*(\BU_+)=\MU_*[b^\MU_n:n\geq1]=\MU_*[m^\MU_n:n\geq1]
\]
(where $m^F_n=\tilde{b^F}_n$) as giving rise to the scheme (on the
category of commutative rings) which evaluated at $R$ gives all
pairs $(F,f(X))$ consisting of a formal group law $F$ over $R$
together with a power series
\[
f(X)=X+\sum_{1\leq n}c_nX^{n+1}.
\]
Interpreting this as $\Alg_\Z(\MU_*(\BU_+),R)$, we identify
the element $\phi$ of the latter with the sequence
$\{\tilde{\phi(b^\MU_n)}\}$.

We may apply the ideas of Section~\ref{sec:4} to describe the
$\evMU_*$-Hopf algebra $\evMU_*(\BU_+)$ equipped with the canonical
odd formal group law $\evFGL$. By the discussion following
Theorem~\ref{thm:2.3} together with the Atiyah-Hirzebruch spectral
sequence, in this case we have
\begin{align*}
\evMU_*(\BU_+)^+&=\evMU_*(\BSp_+),
\\
\evMU_*(\BU_+)^-&=\evMU_*(\Sp/\U_+)
\end{align*}
as Hopf algebras over the $\Z[1/2]$-algebra $\evMU_*$. This gives
the following interpretation of the scheme (on $\Z[1/2]$-algebras)
represented by $\evMU_*(\BU_+)^-$: points in $R$ are pairs $(F,\phi)$
consisting of an odd formal group law $F$ over $R$ together with a
strict isomorphism $\phi\:F_\phi\lra F$, where
\[
\phi(X)=u(X)X
\]
for some series $u\in\mathbf W^{F-}(R)$. Given two such pairs
$(F,\phi),(F,\theta)$, having the same formal group law $F$, their
`composition' is given by the operation of `addition' defined by
$\diamond$. So if $\theta(X)=v(X)X$, we have another such pair
\[
(F,\phi)\Dia{F}(F,\theta)=(F,(u\Dia{F}v)(X)X).
\]
Thus we see that $\evMU_*(\Sp/\U)$ represents a groupoid scheme,
with objects over a $\Z[1/2]$-algebra $R$ being odd formal group
laws over $R$, and morphisms being such strict isomorphisms.

\section{An interpretation in terms of homogeneous spaces}
\label{sec:6}

In this section we sketch an interpretation of some of our preceding
constructions in terms of the geometry of homogeneous spaces. For a
graded $\Z[1/2]$-algebra $R=R_*$ we will write $G(R)$ for the group
of invertible power series with coefficients in $R$, i.e.,
\[
\G(R)=\left\{g(T)=
\sum_{0\leq k} g_{k}T^{k+1}:g_{k}\in R_{2k}, g_{0}=1\right\}
\]
with composition as its operation. We also set
\[
\L(R)=\Alg_{\Z}(\MU_*,R),
\]
which by \cite{Qu} can be identified with the set of (graded) formal
group laws defined over $R$. There is a right action
\begin{align*}
\L(R)\times\G(R)&\lra\L(R);
\\
(F,g)&\longmapsto F^{g}
\end{align*}
defined by changing coordinates via
\[
F^{g}(X,Y)=g^{-1}(F(g(X),g(Y)).
\]
The subgroup
\[
\Godd(R)=\left\{g(T)\in \G(R):g(T)=\sum_{0\leq k}g_{2k}T^{2k+1}\right\}
\]
of odd power series, and the subset $\Lodd(R)$ of odd formal group laws,
are defined analogously; thus $\FGL(R)$ is the category associated to
the transformation group $(\G,\L)$, while $\evFGL(R)$ is defined by the
action of $\Godd$ on $\Lodd$. Finally, we define $\Gev(R)$ to be the
\emph{set} of all even power series over $R$ with zero constant term.

It is now easy to see that the action of $\G$ on the subobject
$\Lodd$ of $\L$ factors through a map
\[
\Lodd\Times{\Godd}G\lra\L
\]
which is an equivalence over $\Z[1/2]$. To see this, first recall
that we have seen that under these circumstances any formal group
law is isomorphic to some odd formal group law, implying that the
map is surjective. On the other hand, if two pairs $(F_1,g_1)$
and $(F_2,g_2)$ (with $F_1$ and $F_2$ both odd) map to the same group
law $F_1^{g_1}={F_2}^{g_2}$, then $h=g_2\o g_1^{-1}$ is an isomorphism
from the odd group law $F_1$ with the odd group law $F_2$, and therefore
commutes with the automorphisms
\[
[-1]_{F_1}(T)=[-1]_{F_2}(T)=-T,
\]
implying that $h$ is odd, and that the equivalence classes $[F_1,g_1]$
and $[F_2,g_2]$ are equal. Hence this map is also injective.

This description of $\L$ yields a diagram of functors
\[
\begin{CD}
\ds\Lodd@>>> \L = \Lodd\Times{\Godd}\G \\
@.                            @VVV   \\
\ds@.     \G\Times{\Godd}\text{pt}=\G/\Godd
\end{CD}
\label{eqn:6.1}
\]
for $\L$ as a $\Godd$-equivariant fibre bundle; but the map $\Lodd\lra\L$
of Section~\ref{sec:2} induced from the ring homomorphism $\MU_*\lra\evMU_*$
is a retraction of the total space to the fibre, which can be interpreted
as a product splitting
\[
\L\lra\Lodd\times\G/\Godd.
\]

If we now define the set of \emph{involutive power series over
$R$} by
\[
\Invol(R)=
\left\{e(T)=
\sum_{i\geq0}e_{i}T^{i+1}:e_{0}
=-1,\;e_{i}\in R_{2i},\;e(e(T))=T\right\},
\]
then a similar construction to the above defines the
$\G$-equivariant map
\begin{align*}
\G/\Godd &\lra\Invol;
\\
g&\longmapsto g^{-1}(-g(T))
\end{align*}
of representable functors. The inversion formula derived in
Section~\ref{sec:1} for the equation
\[
T + e(T)=\sum_{i\geq 1} c_{2i-1}(Te(T))^{i}
\]
expresses the coefficients $c_{2i-1}$ as functions on the space
of involutions; conversely, Hensel's Lemma, Theorem~\ref{thm:HL},
applied to the equation
\[
H(E)=\sum_{i\geq1}c_{2i-1}(TE)^i - T - E  =  0
\]
over the ring $R[[T]]$ and the ideal $I=(T)$, yields an expression
for the involution $e(T)$ as a formal power series in terms of the
coefficients $c_{2i-1}$. Hence these coefficients generate the ring
of homogeneous functions on the scheme of all such formal involutions.

Now recall the group schemes $\mathbf W^-$ and $\mathbf W^{F-}$ of
Section \ref{sec:4}.
\begin{thm}
\label{thm:6.2}
There are equivalences of schemes
\[
\mathbf W^-\xrightarrow[\Phi]{\iso}\Gev
\xleftarrow[\Phi^F]{\iso}\mathbf W^{F-},
\]
where $F$ is an odd formal group law.
\end{thm}
\begin{proof}
The first equivalence is given by assigning to the series
$c=\{c_n\}\in\mathbf W^-(R)$ the even series
\[
\sum_{k\geq1}c_{2k-1}X^{2k},
\]
this gives an bijection $\mathbf W^-(R)\iso\Gev(R)$ since the
$c_{2k-1}$ determine a unique element $\{c_n\}$ of $\mathbf W^-(R)$.
The second equivalence follows from the equivalences
\[
\mathbf W\iso\tilde{\mathbf W}\iso\mathbf W^F
\]
of Theorem~\ref{thm:4.2}, together with that of Theorem~\ref{thm:4.3}.
\end{proof}

We next show that the decomposition of $\L$ into a product can
be given the structure of a family of abelian group objects
parameterized by $\Lodd$. Suppose that $F$ is an odd formal group
law defined over a $\Z[1/2]$-algebra $R$, and that
$u\in\mathbf W^{F-}(R)$; then there is a corresponding element
$\Phi^F(u)\in\Gev(R)$ and we also have the series
$u(T)={\dsum}^F_{n\geq0}c_nT^{n+1}$. Define a new formal group
law $F_u$ over $R$ by requiring the series
\[
\phi_u(T)=T\Plus{F}u(T)
\]
to be a strict isomorphism $\phi_u\:F_u\lra F$.
\begin{thm}
\label{thm:6.3}{\ }

\begin{enumerate}
\item[a)]
The function
\[
\Gev(R)\lra\Invol(R);
\quad
u(T)\longmapsto[-1]_{F_{u}}(T),
\]
is a bijection.
\item[b)]
The map
\[
\G(R)/\Godd(R) \lra\Invol(R);
\quad
g\longmapsto g^{-1}(-g(T)),
\]
is an isomorphism. Equivalently, $\G(R)$ acts transitively
on the set of all involutions $\Invol(R)$.
\end{enumerate}
\end{thm}
\begin{proof}{\ }

\noindent
a) We may write
\[
[-1]_{F_{u}}(T)=(-T)\Plus{F}w(T)
\]
where $w(T)=\dsum_{k\geq2}w_kT^k$ for some $w_k\in R_{2k-2}$.
On applying $\phi_u$ and remembering that $F$ is odd, we obtain
\begin{align*}
(-T)\Plus{F}(-u(T))&=[-1]_F(\phi_u(T)) \\
&=(-T)\Plus{F}w(T)\Plus{F}u((-T)\Plus{F}w(T)),
\end{align*}
since $F$ is odd, and this yields
\[
w(T)=-\(u(T)\Plus{F}u((-T)\Plus{F}w(T))\).
\]
Comparing coefficients of $T^2$ and $T^3$ gives $w_2=-2u_2$ and
expresses $w_3$ as a polynomial over $\Z[1/2]$ in $w_2$ and $u_2$
together with the coefficients of the formal group law $F$. By
induction and comparison of the coefficients of $T^{2k}$ and
$T^{2k+1}$, we find that
\begin{align*}
w_{2k}&=-2u_{2k}+P_{2k}, \\
w_{2k+1}&=P_{2k+1},
\end{align*}
where $P_{n}$ is a polynomial over $\Z[1/2]$ in the $u_{2i}$ ($2i<n$)
and $w_j$ ($j<n$) together with the coefficients of the formal group
law $F$. Now these equations may be inverted to express the $u_{2k}$
in terms of the $w_\ell$,
\[
u_{2k}=(-1/2)\(w_{2k}-P_{2k}\).
\]
Hence, given the series $w(T)$ there is exactly one even power
series $u(T)$ for which
\[
[-1]_{F_u}(T)=(-T)\Plus{F}w(T).
\]

\noindent
b)
Notice that
\[
g^{-1}(-g(T))=-T\iff g(-T)=-g(T)\iff\text{$g(T)$ is odd}.
\]
Hence, this map is injective.

To prove surjectivity, we use an idea from the proof of part~(a).
Let $e(T)\in\Invol(R)$. Taking $F$ to be any odd formal group
law over $R$ (e.g., $F=\hat{\mathrm{G}}_{\mathrm{a}}$, the additive
group law), we can find an even series $u(T)$ with no constant
term which satisfies
\[
[-1]_{F_u}(T)=e(T).
\]
Since
\[
[-1]_{F_u}(T)=\phi_u^{-1}(-\phi_u(T)),
\]
we see that for $g=\phi_u$,
\[
e(T)=g^{-1}(-g(T)).
\]

This completes the proof of Theorem~\ref{thm:6.3}.
\end{proof}

\section{Formal involutions and symmetric spaces}\label{sec:7}

We continue to assume that $R=R_*$ is a graded commutative
$\Z[1/2]$-algebra. Let $f(T)\in R[[T]]$ with
\[
f(T)=T+(\text{higher order terms});
\]
we call such series \emph{strictly invertible} over $R$. We
define the series
\[
e_{f}=f^{-1}\o[-1]\o f,
\quad\text{i.e., }\quad e_{f}(T)=f^{-1}(-f(T))
\]
where $[-1](T)=-T$. Notice that the $e_f\o e_f=\id$ in the
sense that $e_f(e_f(T))=T$.
\begin{thm}\label{thm:7.1}
Suppose $f$ and $g$ are strictly invertible over $R$, and that
the involutions $e_{f}$ and $e_{g}$ commute under composition.
Then $e_{f}=e_{g}$.
\end{thm}

The proof will require the following Lemma. We will say that
a series $e(T)\in R[[T]]$ is a \emph{formal involution} if
$e\o e=\id$, i.e., $e(e(T))=T$.
\begin{lem}
\label{lem:7.2}
Suppose $e(T)$ is a formal involution which strictly invertible
over $R$. Then $e=\id$, i.e., $e(T)=T$.
\end{lem}

\begin{proof}
By Lagrange Inversion of Theorem \ref{thm:LI}, the coefficients
of the series $e(T)=T+\dsum_{n\geq1}e_nT^{n+1}$ satisfy
\[
e_n=-e_n+E_n
\]
for some polynomial $E_n$ over $\Z[1/2]$ in the coefficients
$e_1,\ldots,e_{n-1}$. By induction on~$n$, and remembering 
that $1/2\in R$, we obtain $e_n=0$ for $n\geq1$, where the 
initial case follows from the equation $e_1=-e_1$.
\end{proof}

Now suppose that $e_{f}\o e_{g}=e_{g}\o e_{f}$ for two strictly
invertible series $f$ and $g$. Expanding this equation and
suppressing the composition symbols gives
\[
f^{-1}[-1]fg^{-1}g=g^{-1}[-1]gf^{-1}[-1]g.
\]
On setting $h=fg^{-1}$, this becomes
\[
(h[-1]h^{-1}[-1])(h[-1]h^{-1}[-1])=\id,
\]
which implies that $h[-1]h^{-1}[-1]$ is a formal involution.
Since $h[-1]h^{-1}[-1](T)$ is clearly strictly invertible,
Lemma~\ref{lem:7.2} now gives $h[-1]h^{-1}[-1]=\id$ and so
$h[-1]=[-1]h$, implying that $h=fg^{-1}$ is odd. Thus $f=h\o g$
and $g$ lie in the same left coset of $\Godd(R)$ in $\G(R)$,
i.e., $f\Godd(R)=g\Godd(R)$ in $\G(R)/\Godd(R)$. Hence
$e_{f}=e_{g}$, which proves Theorem~\ref{thm:7.1}. Our next
result shows that $G(R)/\Godd(R)$ is a kind of symmetric 
space.
\begin{cor}\label{cor:7.3}
At every point $e\in \G(R)/\Godd(R)$ there is an involution
\[
\G(R)/\Godd(R)\lra\G(R)/\Godd(R);
\quad
c\longmapsto e^{-1}\o c\o e,
\]
which has $e$ as the unique fixed point.
\end{cor}

There is thus some similarity of the quotient $\G/\Godd$ with the
Riemannian symmetric spaces defined by the quotient of the group
of complex points of a simple Lie group by a maximal compact subgroup.
This fits in with the observation that the Lie algebra of $\G$ is
twice as big as the Lie algebra of $\Godd$, so that one can try to
think of $\Godd$ as the `real' points of some `Lie' group, with $\G$
as the `complex' points.

\section{A Hopf-Galois interpretation}\label{sec:HopfGal}

Although this viewpoint plays no r\^ole in our work, we 
feel it worth mentioning that Rognes's Hopf-Galois theory 
for commutative $S$-algebras~\cite[chapter~12]{JR:Opusmagnus} 
provides another framework for the equivalence
\[
\MSp[1/2]\Smash\Sp/\U_+\simeq\MU[1/2].
\]
For this we need to work with commutative $S$-algebras in 
the sense of~\cite{EKMM}. Actually we need to pass to 
Bousfield local $S$-algebras with respect to $S[1/2]$, these 
are equivalent commutative $S[1/2]$-algebras.

First we note that $\MU$ is such a Hopf-Galois extension 
of the sphere $S$ with respect to the commutative $S$-Hopf 
algebra $S[BU]=\Sigma^\infty BU_+$. Here the Thom diagonal 
provides a multiplicative equivalence 
\[
\MU \lra \MU\wedge\Sigma^\infty BU_+ 
             \xrightarrow{\;\sim\;} \MU\wedge\MU.
\]
The canonical fibrations of infinite loop spaces $BSp\lra BU$
and $BU \lra BSp$ compose to give an equivalence after 
inverting~$2$ hence there is a splitting of infinite 
loop spaces 
\[
BU[1/2] \simeq BSp[1/2] \times Sp/U[1/2].
\]
Then we can view $\MSp[1/2]$ as the homotopy coinvariants 
of the coaction
\[
\MU[1/2] \lra \MU[1/2]\wedge S[BU] 
                          \lra \MU[1/2]\wedge S[Sp/U[1/2]].
\]
Furthermore, $\MSp$ is itself a Hopf-Galois extension of~$S$ 
for the commutative $S$-Hopf algebra $S[Bp]=\Sigma^\infty BSp_+$.

\end{document}